\documentclass{llncs}

\advance\intextsep by -5pt

\usepackage [a4paper]{geometry}

\usepackage[latin10]{inputenc}

\usepackage[pdftex] {graphicx}

\usepackage {amsmath}

\usepackage {amssymb}

\usepackage{verbatim}

\usepackage[font={small}]{caption}

\long\def\appendix#1{}

\long\def\appendix{}

\long\def\xappendix#1{}

\title{ Unit Grid Intersection Graphs: Recognition and Properties }
\author{Irina Musta\textcommabelow t\u a\inst{1} \and Martin Pergel \inst{2}}
\institute{Berlin Institute of Technology, Institut f\"ur Mathematik, MA 2-2, Stra{\ss}e des 17. Juni 136,\\ 10623 Berlin, Germany; funded by Berlin Mathematical School\\
E-mail: \email{mustata@math.tu-berlin.de}
\and
Department of Software and Computer Science Education (KSVI),\\
Charles University Prague. \\
E-mail: \email{perm@kam.mff.cuni.cz};\\
Funded by the GraDR project, as part of the EuroGiga, under a Czech research grant GA\v{C}R GIG/11/E023.
}

\begin{document}

\maketitle

\begin{abstract}

It has been known since 1991 that the problem of recognizing grid intersection graphs is NP-complete. Here we use a modified argument of the above result to show that even if we restrict to the class of unit grid intersection graphs (UGIGs), the recognition remains hard, as well as for all graph classes contained inbetween. The result holds even when  considering only graphs with arbitrarily large girth. 
Furthermore, we ask the question of  representing UGIGs on grids of minimal size. We show that the UGIGs that can be represented in an $(1+\varepsilon) \times (1+\varepsilon)$, for $0<\varepsilon\leq 1$ grid size are exactly the orthogonal ray graphs, and that there exist families of trees that need an arbitrarily large grid. 

\end{abstract}

\newtheorem{prop}{Proposition}
\newtheorem{thm}{Theorem}
\newtheorem{lem}{Lemma}
\newtheorem{cor}{Corollary}

\section{Introduction} \label{sec:intro}


A graph $G=(V,E)$ is said to be an intersection graph of a set-system
${\cal S}\subseteq 2^{\cal M}$ for some set $\cal M$ if each vertex $v\in V$
can be represented by a set $s_v\in \cal S$ such that $uv$ is an edge of $G$
if and only if $s_u$ and $s_v$ have a non-empty intersection. As each graph
has an intersection representation (with a suitable set {\cal M} and set-system
{\cal S}), it is interesting to restrict the
set-system $\cal S$, as well as the base set $\cal M$. A widespread type of such restrictions consists in assuming $\cal M$ to be the plane and $\cal S$ a collection of topologically or geometrically defined objects in the plane. If the objects are arc-connected, one can speak about geometric intersection graphs  

Geometric intersection graphs have been explored because an appropriate
intersection representation often permits the design of efficient algorithms for
generally (NP-)hard problems. Applications of these intersection graphs can
be found not only in VLSI-circuit design and compiler-construction, but
also in archaeology or even in the modeling of ecosystems \cite{MM}.

The advantage of working with (geometric) intersection graphs is that it permits working with elements of a representation, which often exhibit structure that allows the development of efficient algorithms. The recognition of such graphs is of particular interest, as well as the relationship between classes of geometric graphs modelled as intersection graphs of similar types of objects.

\subsection{Notions and definitions} \label{ssec: notdef}

One of the best understood classes of geometrically-defined intersection
graphs is that of interval graphs \cite{Gil} (representable as intersection graphs of intervals on a line). For this class, the maximum independent set and maximum clique can be found in polynomial
time, as well as the chromatic number, since interval graphs are perfect.  This class  can also be recognized in linear time \cite{Boo,HMPV}.

Another well-known class is the class of string graphs, intersection graphs
of arc-connected curves in a plane. For string graphs, clique, colorability,
independent set and several other problems are NP-complete. The recognition is also
NP-complete \cite{KratII,Pach,SSS}. When restricting the curves shape to
straight line segments, the resulting class is that  of segment graphs (SEG). Despite the extra structure information, the problems of finding the chromatic numbers or a maximum independent set remain NP
-hard. Almost the same holds for the recognition problem; it is NP-hard but it is not known whether it is in NP. A graph class intermediary between string and SEG is that of pseudosegment graphs (PSEG). It is obtained using "topological segments", i.e.,  curves such that each
pair  intersects at most once \cite{Fel,KM,ChGO}.
The recognition problem for this class is already known to be NP-complete.

Due to the complexity of the latter three classes, special subclasses have been
defined. One such instance occurs when allowing the straight line segments to use 
only one of $k$ prescribed directions, in which case the corresponding graphs are called k-directional segment ($k$-DIR) graphs \cite{KratII}.
This does not simplify the recognition problem, since even 2-DIR graphs are hard to recognize \cite{KratII}.

One restriction that can be applied particularly to 2-DIR graphs, is the prohibiting of  intersections for any two collinear segments. The 2-DIR graphs with this property are bipartite, with one partition representable by (w.l.o.g.) horizontal, the other by vertical segments. The class is called PURE 2-DIR or grid intersection graphs (GIG). In the sequel,  the latter terminology will be used. This class can be further restricted
by prescribing each segment to have constant (unit) length, in which case one obtains the
class of unit grid intersection graphs (UGIG), which makes the focus of this paper. Sandwiched between GIG and UGIG, we can define
a class USEG (unit-segment), of grid intersection graphs where the unit-length restriction applies only to the segments representing one partition.  

Obviously, it holds that $\hbox{UGIG}\subseteq\hbox{USEG}\subseteq\hbox{GIG}
\subseteq \hbox{SEG}\subseteq\hbox{PSEG}$. All the inclusions are proper. For the
first two inclusions we show the proof in Section \ref{sec:relothers}.

Another particular case occurs when a GIG has a representation where the segments could be extended into half-lines, without creating additional intersections. This class is called \emph{orthogonal ray graphs} (ORGs), and 
it was introduced in \cite{STU}, together with the observation that they constitute a subclass of UGIGs.

Considering graph-optimization problems like, e.g., maximum clique or chromatic number,
on two graph-classes ${\cal A} \subseteq \cal B$, an efficient algorithm
for class $\cal B$ often provides an efficient algorithm for class $\cal A$

This is not the case with the recognition problem. A polynomial algorithm for recognizing a superclass
provides no information about the recognition of one if its subclasses or conversely.

Since there are many intersection-defined graph classes, answering the recognition
problem for individual classes becomes inefficient. 
Therefore, Kratochv\'\i l with several coauthors \cite{KratII,KM,HK}
started designing polynomial reductions that can show hardness
of recognition for any class $\cal C$ such that ${\cal A}\subseteq {\cal C}
\subseteq \cal B$, where $\cal B$ is a known hard class. The first reduction of this type showed the hardness
everywhere between segment- and string-graphs (hence, e.g., pseudosegment graphs).
Later this concept gets named {\em sandwiching} \cite{KP,Perm}.

When recognizing intersection graphs, the density of edges seems to play an important role (providing the information about the presence of large cliques or about the girth, i.e. length of the shortest cycle).
The article \cite{KP} shows that there are graph classes
for which low edge  density may change the recognition class (polygon-circle graphs), but also classes that remain hard to recognize even with arbitrary girth (segment- and
pseudosegment-graphs).

In the article we explore particular structural properties of UGIGs (optimal size of representation)
and their relation to other graph-classes, and conclude by showing that this class is NP-complete
to recognize.  This class remains NP-complete
even with arbitrary girth and, moreover, no polynomially recognizable
class can be found between UGIGs and PSEG.

\subsection{Paper structure}

After having introduced  the relevant definitions of graph classes in the beginning of this section, we will proceed by summarizing the  known results about UGIGs and related classes. The second section of the paper highlights the theoretical context of UGIGs, showing (and reminding, respectively) that UGIGs are a proper subset of GIGs, and that the proper character of inclusions is maintained even after introducing the hybrid class USEG. Furthermore, we reproduce the argument that shows the containment of ORG in UGIG, and conclude with noting that all trees have a unit grid representation.

Section 3 is divided into two parts. The first is concerned with providing an upper bound (in terms of the number of vertices) for the  grid size needed to acommodate a UGIG representation. The second part is a joint work with Stefan Felsner and addresses the question of describing  UGIGs that fit in grids of fixed sizes, or whether at all such bounds exist. We prove that the unit grid intersection graphs that can be drawn in a $(1+\varepsilon) \times (1+\varepsilon)$ size square with $0<\varepsilon < 1$ arbitrary, are exactly the orthogonal ray graphs and, at the same time, there exist families of trees that require an unboundedly large grid size.

Section 4 contains the main result with respect to the recognition and description of UGIGs. Here we modify an existing argument that shows the hardness of recognizing (2-DIR)segment graphs, which yields  a proof for the NP-hardness for all graph classes lying between UGIG and GIG, hence UGIG and USEG in particular.

\subsection{Previously known results} \label{sec:prevres}

The class of grid intersection graphs was introduced by I. Ben-Arroyo Hartman, I. Newman and R. Ziv, in \cite{HNZ}. It has subsequently been shown in \cite{BHPW} that it can be described as the intersection class between bipartite graphs and graphs of boxicity 2. In the same work, GIGs were described as the graphs whose bipartite adjacency matrix is cross-freeable, i.e. there exists a permutation of the lines and columns such that the resulting matrix is cross-free, where a cross is defined as $ \left( \begin{array}{ccc} \ast & 1 & \ast \\ 1 & 0 & 1  \\ \ast & 1 & \ast \end{array} \right)$.  Moreover, it was shown in  \cite{HNZ} that all bipartite planar graphs admit a grid intersection representation. The hardness of recognizing GIGs was established by J. Kratochv\'il in \cite{KratII}, whereas in \cite{KP}, he shows together with M. P.  that for any fixed $k$, $k$-DIR and pure $k-$DIR are hard to recognize. Unit grid intersection graphs were subsequently analyzed by Y.Otachi, Y. Okamoto and, K. Yamazaki in \cite{OOY}, where it 
is concluded that they are a proper subclass of GIGs, and strictly including  $P_5$-free 
bipartite graphs and bipartite permutation graphs, as a consequence of these being included in ORGs. In \cite{UEH}, R. Uehara establishes that some important problems, such as the existence of a Hamiltonian cycle, and the graph isomorphism problem, remain hard even when restricted to UGIGs and GIGs, respectively. Another analyzed subclass of UGIG (and of ORG) is that of two-directional orthogonal ray graphs (2-DORGs) that were introduced in \cite{STU} by A.S. Shrestha, S. Tayu and S. Ueno. Unlike what we show here about UGIGs, their recognition is in $P$. \cite{STU}. This result follows from the several models that can be used to describe 2-DORGs. They are the complement of the quadratic time recognizable class co-bipartite circular arc graphs \cite{STU,FHH}. At the same time, they are also exactly the comparability graphs of posets of height 2 and interval dimension 2, which have an efficient recoginition and can be described in terms of forbidden subgraphs \cite{FHM,FEL2}.

\section {Relationship with other graph classes} \label{sec:relothers}

\subsection {Comparison with ORG} \label{ssec:comporg}

Here we remind the following result from \cite{STU}:

\begin{prop}\label{prop:orginugig}

Orthogonal ray graphs are a  subset of unit grid intersection graphs.

\end{prop}

\begin{proof}

Let $G$ be an ORG, together with a fixed representation $\mathcal{R}$. Let $\mathcal{L}$ be a square with sides parallel to the coordinate axes that contains all intersection points in $\mathcal {R}$ inside of it. (see the figure below). Then, it is clear that the restriction of $ \mathcal{R}$ to the interior of $\mathcal{L}$  provides a UGIG representation of $G$. 

\begin{figure}[htbp]

\centering

\includegraphics[width=3.5cm]{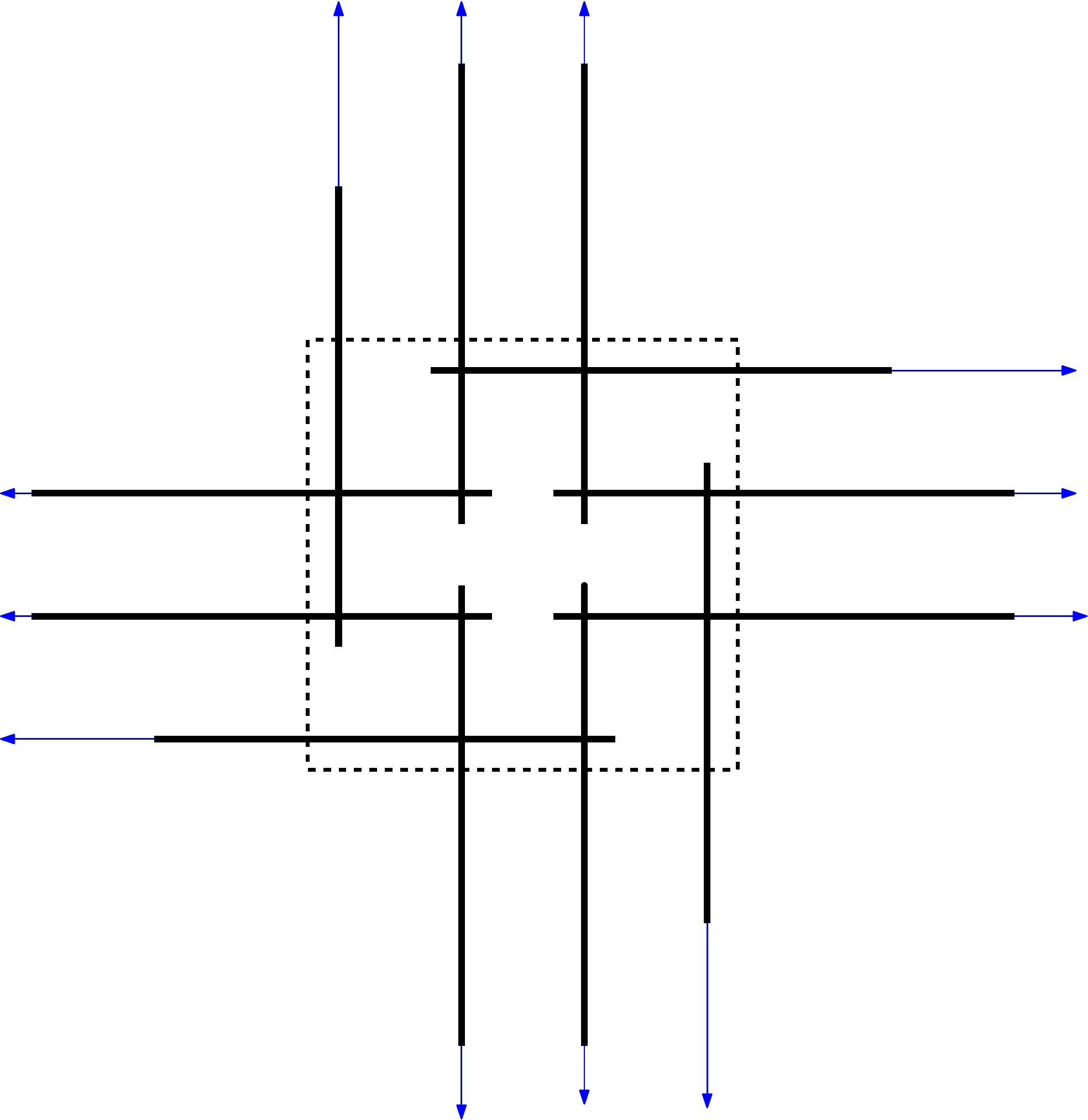}

\caption{ORGs are a subset of UGIGs}

\label{fig:orginugig}

\end {figure}

\end{proof}

In addition to the above result, we can show that
$ORG\subset UGIG$ is a proper inclusion.
To show this, it suffices to prove the following result, also detailed in \cite{STKU}, however using here another proof idea:

\begin{prop}\label{lem: nonorgcycles}

All cycles $C_{2n}, n\geq 7$  do not have an ORG representation.

\end{prop}

\begin{proof}

Assume the contrary, i.e. there exists a $C_{2n}, n\geq 7$ with a (fixed) representation $\mathcal{R}$ as an orthogonal ray graph. Without loss of generality, we concentrate on the (at least) seven vertices represented as vertical halflines. By the pigeon hole principle, there must exist at least four of them, the corresponding halflines of which have the same infinite direction, w.l.o.g. $y$-axis positive. From left to right, we denote these vertices $v_1,v_2,v_3,v_4$. 

A \emph{separating edge} in $\mathcal{R}$ is defined as an $e\in E$, such that there exist $u_1,u_2$ in $V$, with the neighbourhoods of $u_1$ and $u_2$ disjoint from $V(e)$, such that $ \mathcal {R}(u_1)$ and $ \mathcal {R}(u_2)$ are separated in different plane regions by $\mathcal{R}(e)$.

The existence of a separating edge suffices to show $\mathcal{R}$ is no ORG representation. Indeed, assume this were the case. Then, any path connecting $u_1$ and $u_2$ in $G$ would have to contain one of the endpoints of $e$. However, since in this case $G=C_{2n}, n\geq 7$, for every $e, u_1, u_2$ as above there exists a path connecting $u_1$ and $u_2$ disjoint from $e$, contradiction. 

We conclude the proof with the following

\begin{lem}

Any representation $\mathcal{R}$ of $G$ as an ORG contains a separating edge.

\end{lem}

\begin{proof}

We consider the vertices $v_1,v_2,v_3,v_4$ as defined above and concentrate on $v_2$. Its two neighbours must be represented in one of the three ways below:

\begin{itemize}

\item In case a), the neighbours of $v_2$ are represented by halflines pointing in opposite directions. Then, $v_1$ is separated by the edge going to the right from the vertex among $v_3,v_4$ not belonging to the edge.
\item In case b), both neighbours of $v_2$ correspond to halflines pointing to the left. An edge among these two  is not incident to $v_1$ and separates it from $v_3$ or $v_4$. 
\item In case c), both neighbours of $v_3$ are drawn as halflines pointing to the right. As above, an edge among these two is not incident to $v_1$ and separates it from its non-endpoint among $v_3$ and $v_4$.

 \end{itemize}

 \begin{figure}[htbp]

\centering

\includegraphics[width=10cm]{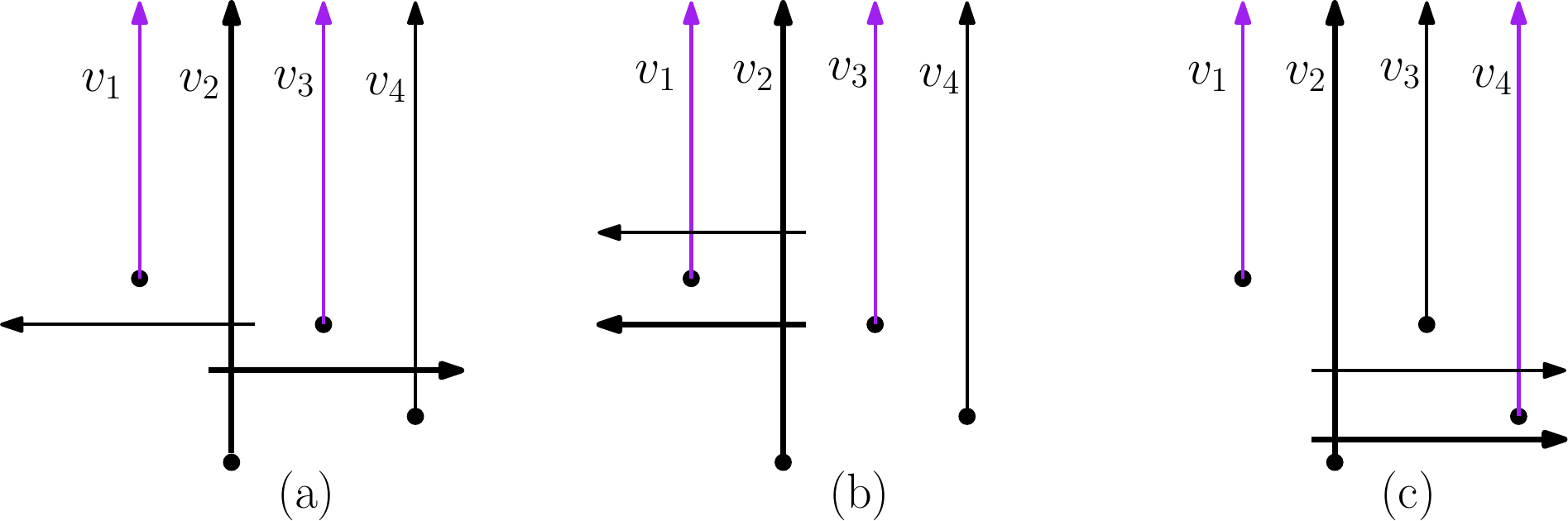}

\caption{The possible representations of the neighbours of $v_2$}

\label{cycnotorg}

\end {figure}

\end{proof}

\end{proof}

\subsection{Comparison with USEG}\label{ssec: compuseg}

\subsubsection{ UGIG and USEG} \label{sssec: ugiguseg}

It has been shown in \cite{OOY} that there are graphs admitting a grid representation that cannot be represented with unit (constant length) segments. We will here provide another such example of another, smaller graph which we will denote $\mathcal{S}$, which in fact even admits a segment-unit representation, hence proving the proper inclusion of UGIG in USEG. 

\begin{figure}[htbp]

\centering

\includegraphics[width=13cm]{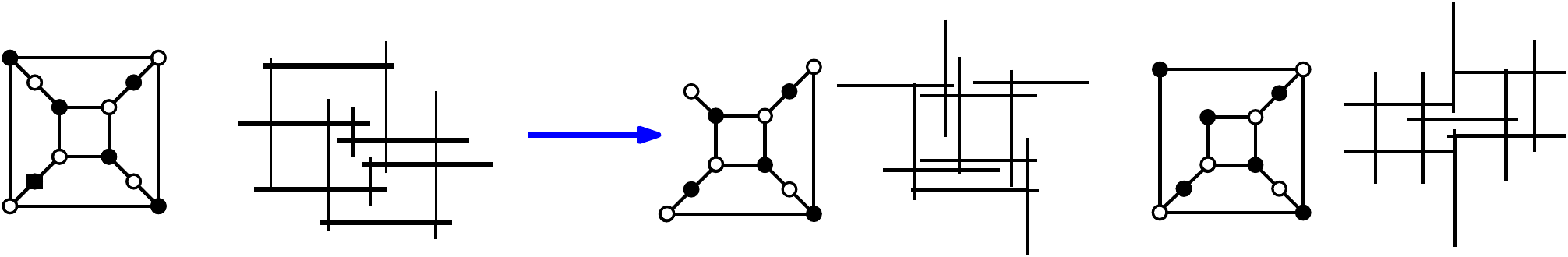}

\caption{Left: The graph $\mathcal{S}$ with a USEG representation. Right: Any induced subgraph of $\mathcal{S}$ is a UGIG}

\label{fig:notugig}

\end {figure}

\begin{prop}

The graph  $\mathcal{S}$ belongs to $\hbox{USEG} \setminus \hbox{UGIG}$. 

\label{prop:properinclusion}

\end{prop}

\begin{proof}

As $ \mathcal{S}$ is clearly bipartite and planar, it follows immediately from \cite{HNZ} that it admits a grid segment representation. In figure \ref{fig:notugig} one can see that it even admits a representation with one partition consisting of unit segments.  Assuming a given representation as an UGIG, we distinguish two cases based on the relative position of the two 4-cycles:

\begin{itemize}

\item One of them is nested inside the other. This is clearly not possible, because, since disconnected, the length of the horizontal (vertical) segments of the inner cycle must be strictly smaller than the length of the horizontal (vertical) segments of the outer cycle. 
\item The two four cycles lie beside each other. We denote the two cycles  $C_1$ and $C_2$, respectively. Without loss of generality, $C_2$ lies to the right of $C_1$. Let $a_1, a_2$ and $b_1, b_2$ denote the horizontal segments of $C_1$ and $C_2$ respectively, numbered in decreasing height order. Note that it is not possible to pair the segments to be connected in $(a_1, b_2)$ and $(b_1, a_2)$ as this gives rise to a forbidden intersection of paths: 

\begin{figure}[htbp]

\centering

\includegraphics[width=5.5cm]{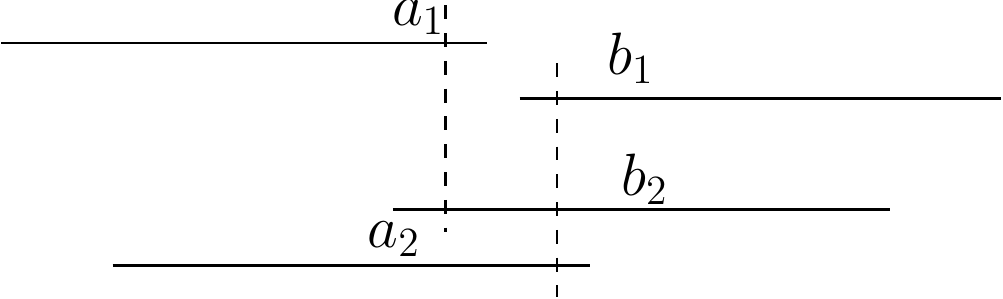}

\caption{Forbidden pairing}

\label{intersecting}

\end {figure}

Hence,the only admissible pairing is $(a_1,b_1)$, $(a_2,b_2)$. However in this case, regardless of the ordering of $a_1,a_2,b_1,b_2$, the connecting vertical segments block both the rightmost vertical segment of $C_2$ and the leftmost vertical segment of $C_1$.

\end{itemize}

\end{proof}

Furthermore, note that the above counterexample is minimal. As depicted on the right side of Figure \ref{fig:notugig}, the removal of any vertex (there are two equivalence classes here, the vertices of degree 2 and those of degree 3) makes a representation as an UGIG possible.

\subsubsection{USEG and GIG}\label{sssec: useggig}

The example from \cite{OOY} mentioned before, i.e., a graph with a grid intersection representation that is not a UGIG, can be shown along the same lines as in the original work not to admit even a USEG representation. 

\begin{figure}[htbp]

\centering

\includegraphics[width=6.5cm]{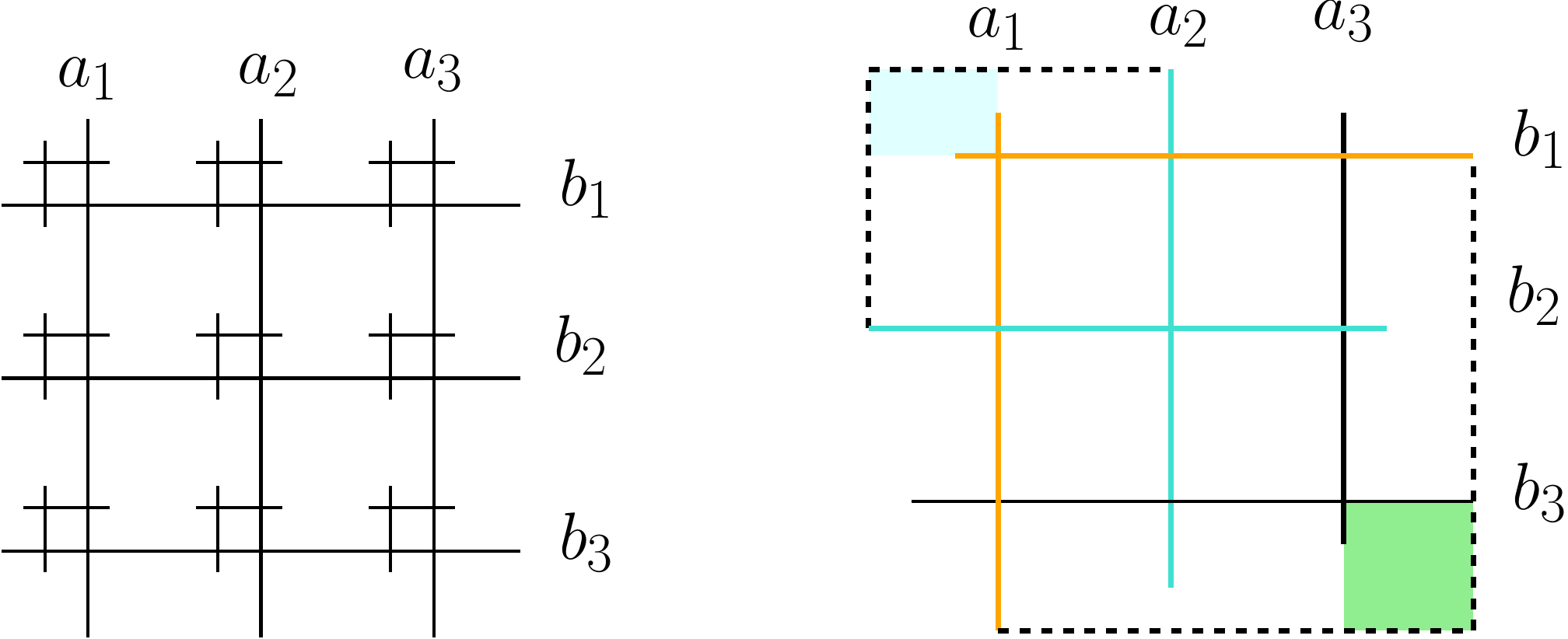}

\caption{GIG strictly includes USEG}

\label{notuseg}

\end {figure}

We  reproduce the argument from \cite{OOY} here, noting that  this time, the conclusion extends over representability as USEG. 

Assume the graph above is representable as a GIG such that for one partition the corresponding segments have constant length. Due to symmetry we can assume without losing  generality that these segments are the horizontal ones. In the given representation $\mathcal{R}$, let $a_1,a_2,a_3$ be the left-right order of the vertical segments of the induced $K_{3,3}$, and $b_1,b_2,b_3$ the top-down order of the horizontal ones.

The path of length two connecting $a_2$ and $b_2$ cannot be represented inside the square delimitated by $\mathcal{R}(a_1),\mathcal{R}(b_3),\mathcal{R}(a_3),\mathcal{R}(b_1)$, due to the horizontal segments having constant length. Hence, the halflines of the two vertices of this path meet over a corner of $\mathcal{R}$, w.l.o.g the top-left corner. But then the path connecting $a_1$ and $b_1$ must be represented over the bottom-right corner, lest the constant length of the horizontal segments be violated. This leaves no place for connecting $a_3$ and $b_3$ with a path of length two, contradiction.

\subsection {UGIG and trees}\label{ssec:ugigtrees}

 We will construct a graphical representation as in \cite{OOY}, where the result is not mentioned, based on which we conclude:

\begin{prop}
All trees can be represented as unit grid intersection graphs.

\label{prop:treesareugigs}

 \end{prop}

 \begin{proof}

  \begin{figure}[htbp]

\centering

\includegraphics[scale=0.25]{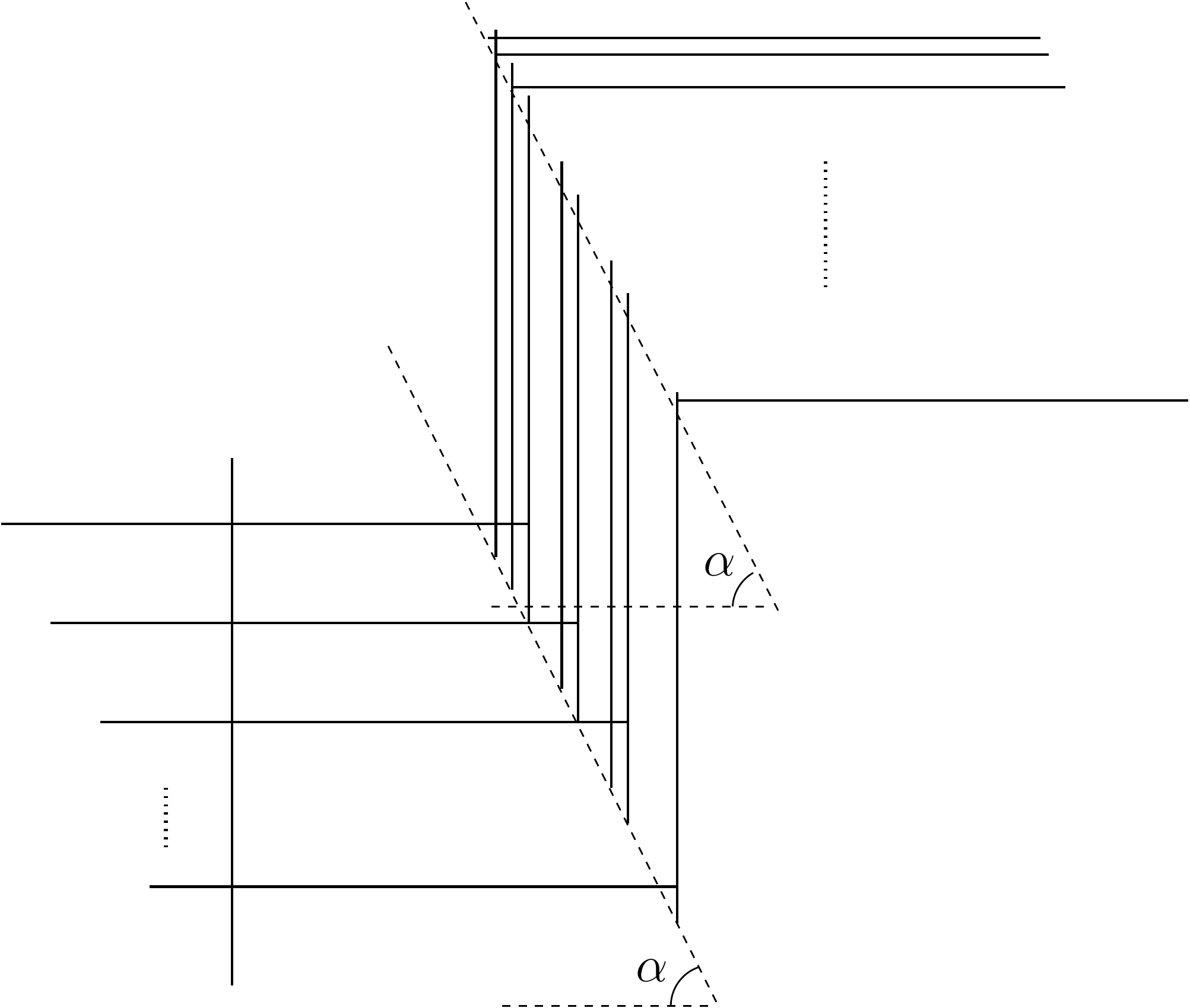}

\caption{General representation of trees as UGIGs}

\label{fig:trees}

\end {figure}

The general construction respects the recursive pattern suggested by picture \ref{fig:trees} above, with a more formal description below.

We consider a fixed slope $\alpha$ that is neither vertical nor horizontal and a given (w.l.o.g) vertical unit segment $r$ as the root. Let $l_1$ be a line  with slope $\alpha$. All children of $r$ are now added as horizontal unit segments the endpoints of which lie on a line parallel to $l_1$. For drawing the elements of each subsequent level $n$ of the tree, one proceeds iteratively, as follows (we present the proof for adding a level of vertical segments, the other case is analogous):

 Let $\varepsilon_n$ be the smallest distance between two $x$-coordinates of the rightmost endpoints of the segments belonging to the last level. We intersect these segments with $l_n$ parallel to $l_1$, such that the (note, constant) distance to the rightmost endpoints of the previous level is $\delta_n < \varepsilon_n$. Each horizontal segment now has a non-zero length piece at the end with a portion of $l_n$ lying directly below. The children of each such vertex can now be drawn as pointing upwards and with the lower endpoint on the segment of $l_n$ lying directly below the representation of the parent.

\end{proof}

\subsection {An upper bound for the area needed to represent an UGIG with $n$ vertices}

 We here investigate the problem of drawing a segment intersection model of a given unit grid intersection graph $G$ in the plane, with the aim of finding an upper bound on the necessary grid size in terms of the number of vertices of $G$. 

 The following holds:

 \begin{prop}

Any unit grid intersection graph on $n$ vertices can be represented
in a grid $(n+1)\times (n+1)$ with all coordinates being multiples
of $\frac{1}{n}$. Moreover, we can find such a representation that
with respect to each axis no two segments have the same
non-integral part of the coordinate.

\label{prop:canon}

\end{prop}

 Note that the lower bound for the granularity (multiples of $\frac{1}{n}$) is tight because
of $K_{1,n-1}$ which requires this precision for the (distinct) coordinates.

 \begin{proof}

 Let us consider an arbitrary UGIG-representation of a given graph $G(H\cup V, E)$, where $H$ denotes the set of vertices to be represented by horizontal segments, and $V$ those to be represented by vertical segments, respectively. We process each axis separately (i.e.  we find a canonical representation for each axis). Details are in the Appendix. 

Without loss of generality, in the sequel we will describe  the procedure for the $x$-axis. Projecting this arbitrary representation of $G$ on the $x$-axis, we obtain a sequence of intervals (corresponding to the vertices in $H$ ) and points (corresponding to the vertices in $V$). By eventually employing small perturbations, it can be assumed that the projected elements are all in general position: no interval is degenerated, no two endpoints coincide (here we treat the projection points stemming from $V$ as endpoints). 

To simplify the following sweeping-argument, we extend individual points
to unit-segments (in an arbitrary fixed  along the $x$ axis). Thus we obtain
an arrangement of unit segments. We build the canonical representation by performing a sweep from  left to  right and employing the following steps:

\begin{itemize}

\item For the left-most segment we assign to the right endpoint the coordinate $0$ (hence the left endpoint now has coordinate $-1$ ). This segment is now the reference from which the rest of the construction will follow.

\item We process the segments in their left to right ordering where each newly added segment  has its left endpoint assigned to the next free 'slot': namely, the smallest multiple of $\frac{1}{n}$ that is free, such that the overlapping (or disjointness) conditions to the previously assigned segments are not violated. This is illustrated in  Figure \ref{fig:sweep} below. Note that due to the chosen granularity it is always possible to find such a free slot. In the worst case, i.e., the segment to be added is disjoint to all the previous ones, the needed increment for the left coordinate is $1+\frac{1}{n}$. This leads to an upper bound of $n+1$ large bounding size per coordinate (which the initially mentioned example fulfills, note however, this is not an optimal representation).

\item In the next step, the projection on the $y-coordinate$ is processed analogously.

\item After the double sweep, all relevant coordinates being known, one can draw the final configuration. Note that the above procedure does not change the relative position of the endpoints and projected segments, hence not contradicting the adjacency and non-adjacency conditions. 

\end{itemize}

\begin{figure}[htbp]
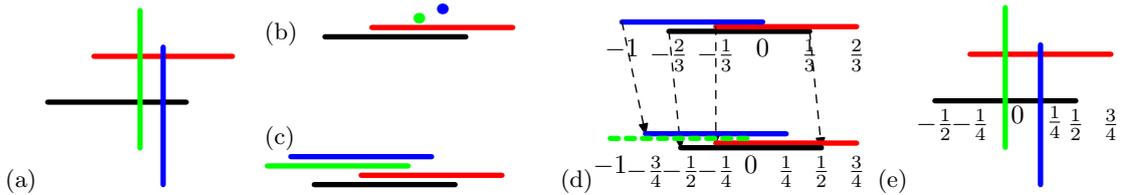


 (a) \includegraphics[width=2.5cm]{obrazky.6}  \hfill
\vbox{\hsize=3.5cm  (b) ~~\includegraphics{obrazky.7}
\vskip 1cm
(c) ~~\includegraphics{obrazky.8}} \hfill 
(d)\includegraphics{obrazky.9}
(e) \includegraphics{obrazky.11}

\caption{Considering the arrangement from picture (a) we make a projection onto one axis (b), i.e., perpendicular segments collapse into single points. We extend them in one direction (in our case to the left), see picture (c). Picture (d) shows the fourth step of the sweeping algorithm, i.e., extending the so far obtained representation. Picture (e) shows the final representation (with respect to $x$-axis, to obtain the desired representation it would be also necessary to sweep along the $y$-axis).}

\label {fig:sweep}

\end{figure}

 \end{proof}

\subsection {Description of UGIGs representable in several fixed-size plane regions}

\subsubsection{The corner structure}\label{ssec:corner}

Consider the union of two intersecting thin rectangular strips, one vertical, the other horizontal, of width $0<\varepsilon<1$ and length larger than 1. We call this region a \emph {corner structure}.  For any UGIG representation with all intersection points within this area, one notes that the vertical and horizontal segments can only lie in the vertical, and horizontal strip, respectively, yielding a situation as in the Figure \ref{corner}.


\begin{figure}[htbp]

\centering

\includegraphics[width=8cm]{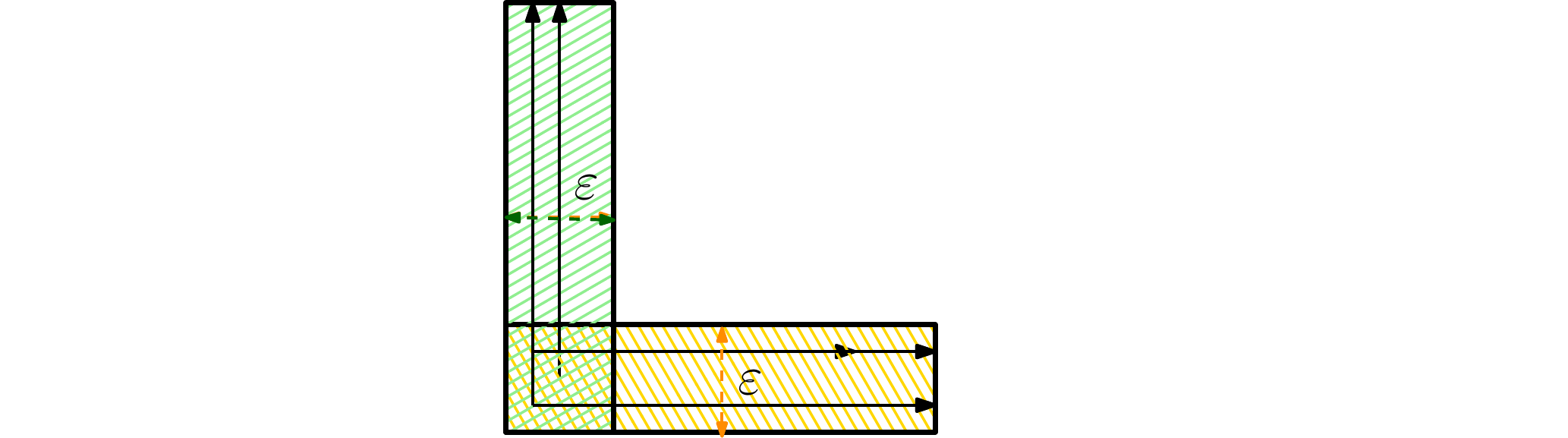}

\caption{The UGIGs fitting in the corner structure}

\label{corner}

\end {figure}

 Note that it is possible to prolonge the vertical(horizontal) segments into halflines with axis-positive infinite directions, thus obtaining a representation as a two-directional orthogonal ray graph. Conversely, every 2-DORG can be placed so that all intersection points fit in a square of side length less than $\varepsilon$, the containing corner structure being derived now immediately.

 Thus we conclude with the following 

 \begin{prop}

 The UGIGs fitting in a corner structure are exactly the two directional orthogonal ray graphs.

 \end{prop}

\subsubsection{The square of area almost 1} \label{ssec: unitsquare}

One can further ask what UGIGs can be drawn into a minimal non-trivial rectangle, i.e. a square of size $(1+\varepsilon)\times (1+\varepsilon)$, with $0<\varepsilon<1$. We will further require, for ease of computation, that the square be open.

We start with the following lemma:

\begin{lem}

Given $G$ a UGIG with a representation fitting inside the open $(1+\varepsilon)\times (1+\varepsilon)$ square, it is possible to alter the drawing so that the central $(1-\varepsilon)\times (1-\varepsilon)$ square in the given representation space remains empty.

\end{lem}

\begin{proof}

Let $v$ be  a vertex of $G$. Then, since it is a unit segment representation, the total overlap of the segments corresponding to its neighbours is the same as the overlap of the two extremal ones. Since the union of these two segments is strictly smaller than the square size, namely $1+\varepsilon$, their intersection must be strictly larger than $1-\varepsilon$. Hence, the segment of $v$ can be slided alongside its neighbours  until reaching either the left or  right (upper and lower respectively) strip of size $\varepsilon$ of the large square. Repeating the procedure for all vertices of $G$, we reach the desired configuration. 

\begin{figure}[htbp]

\centering

\includegraphics[width=4cm]{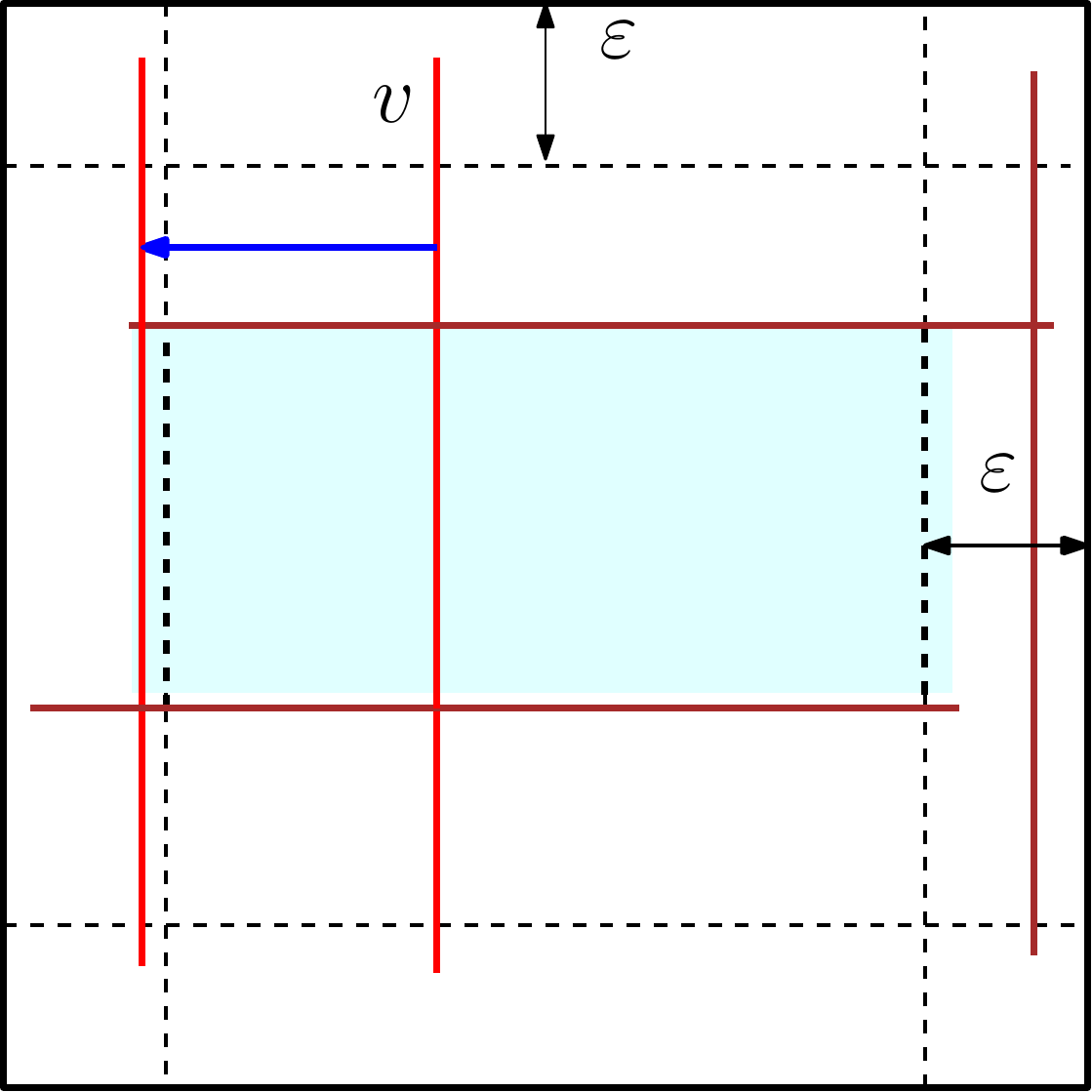}

\caption{Obtaining an $\varepsilon$-frame representation}

\label{epsframe}

\end {figure}

\end{proof}

Note that in the new representation all segment intersections appear in the four $\varepsilon \times \varepsilon$ corners and all the segments can be distinguished according to which strip of the frame they appear. This appears analogous to the two orientations per direction and four intersection types that characterize orthogonal ray graphs, and is no coincidence:

\begin{thm} \label{thm:orgineps}

The UGIGs that can be represented inside an open $(1+\varepsilon)\times (1+\varepsilon)$ square are exactly the orthogonal ray graphs. 

\end{thm}

\begin{proof}

Let $G$ be a graph admitting an orthogonal ray graph representation, with the four partitions corresponding to the infinite directions being $H_l, H_r, V_u$ and $V_d$. One can pick a drawing of $G$ where all intersection points lie inside the square of coordinates $(\delta, \varepsilon-\delta) \times (\delta, \varepsilon-\delta)$ with $0 < \delta < \varepsilon$. 

For each $v \in V(G)$, we now define the following one-to-one correspondence between the endpoints of halflines and segments:

\[ (x_v,y_v) \leftrightarrow \begin{cases}

(x_v+1,y_v)\ldots(x_v+1,y_v+1), \ v\in H_r

\\

(x_v,y_v)\ldots(x_v,y_v+1), \ v\in H_l

\\

(x_v,y_v+1)\ldots(x_v+1,y_v+1), \ v\in V_u

\\

(x_v,y_v)\ldots(x_v+1,y_v), \ v\in V_d

\end{cases} \]

It is a simple exercise to see that the above bijection preserves the graph structure. 

\end{proof}

\begin{figure}[htbp]

\centering

\includegraphics[width=8cm]{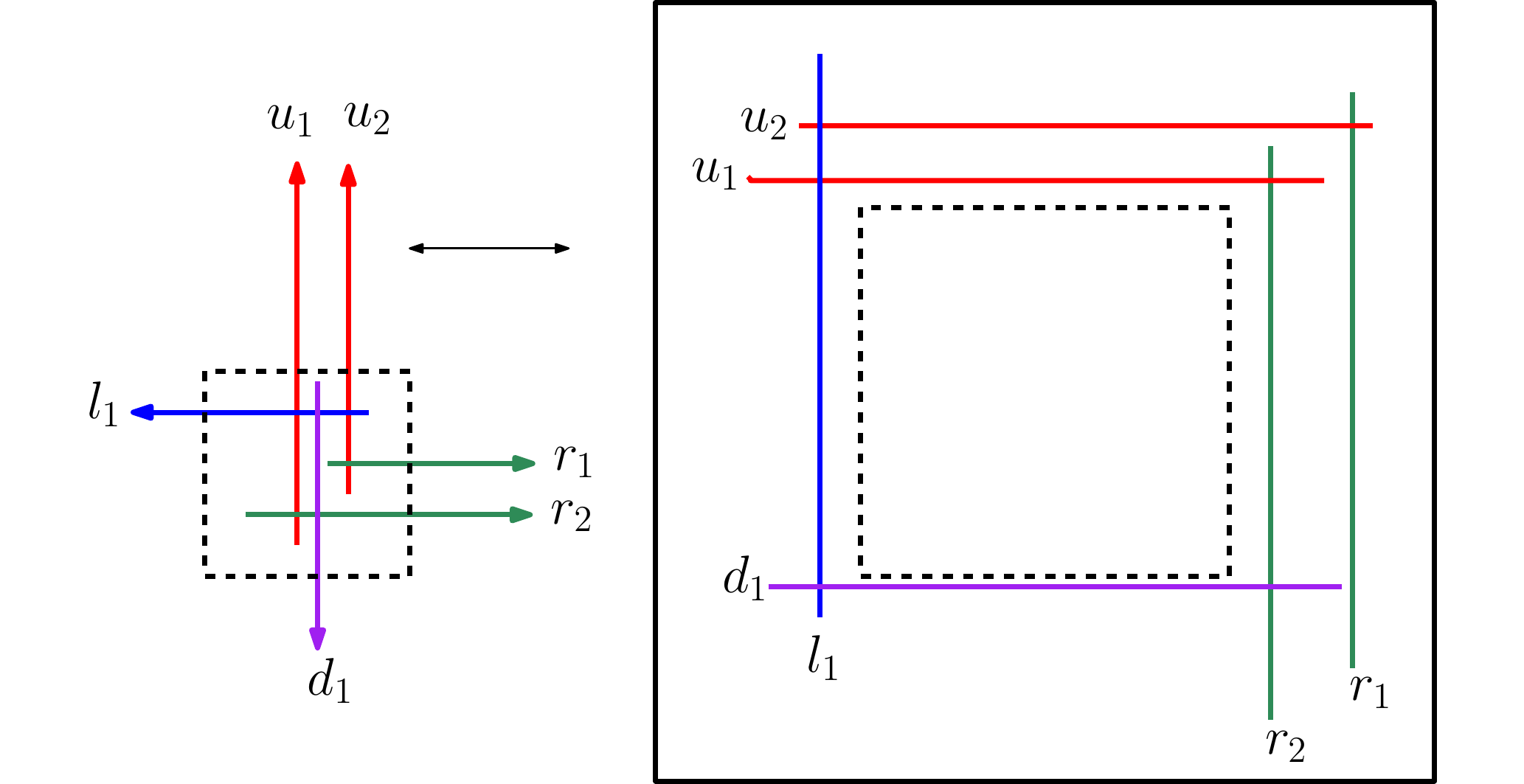}

\caption{Equivalence between ORGs and UGIGs fitting an almost unit square}

\label{orgcorresp}

\end {figure}

\subsubsection{Cycles}\label{ssec: cycles}

All cycles are UGIGs and can be represented inside a rectangle of size $(2+\varepsilon) \times (1+\varepsilon), 0<\varepsilon<1$, as it can be seen in Figure \ref{cycles}.


\begin{figure}[htbp]

\centering

\includegraphics[width=3.5cm]{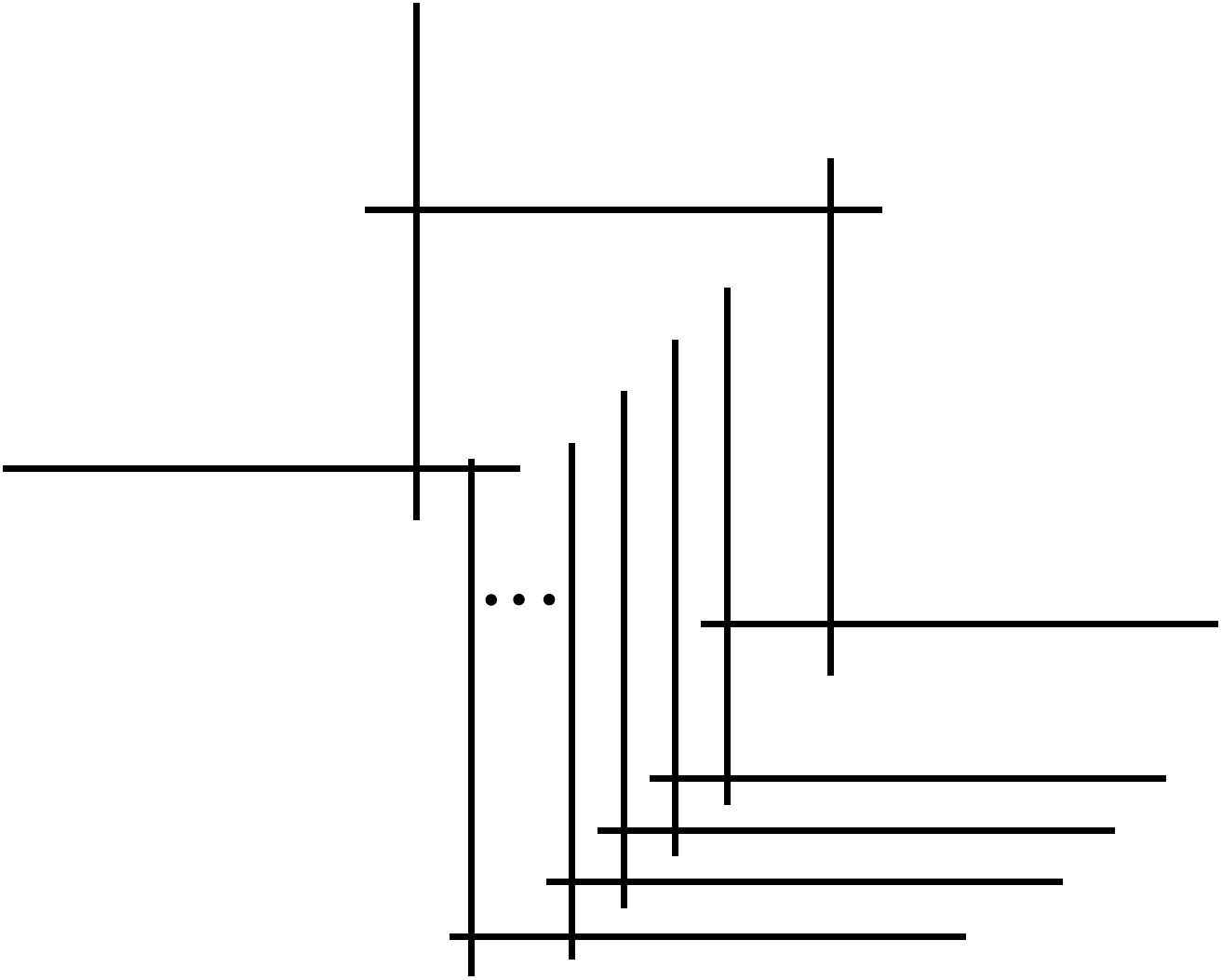}

\caption{Cycles as UGIGs}

\label{cycles}

\end {figure}

 According to proposition ~\ref{prop:orginugig} from section ~\ref{ssec:comporg}, for all $\mathcal{C}_{2n}$ with $n \geq 7$, this grid size is also necessary, since only orthogonal ray graphs fit in an $(1+\varepsilon)\times (1+\varepsilon)$  square. 

 Figure ~\ref{fig:orginugig}, depicting $\mathcal{C}_{12}$, also suggests how to represent the smaller cycles as UGIGs fitting in a $(1+\varepsilon) \times (1+\varepsilon)$ square: at every desired step, delete a concave corner and prolonge two circularly oriented half-lines until they meet. 


 \begin{prop}

 A cycle $C_{2n}$ can be optimally represented in a grid of size $(1+\varepsilon) \times (1+\varepsilon)$ if and only if $n\leq 6$, otherwise a $(2+\varepsilon) \times (1+\varepsilon)$ rectangle is necessary and sufficient. 

 \end{prop}

\subsubsection{Trees and unboundedness}\label{ssec:unboundtrees}

 Above, we have seen UGIG subclasses that fit in fixed-size rectangles. A natural question is whether there is a grid size which accomodates all unit grid intersection graphs. We will show that the answer is negative even when restricting our analysis to trees.

  In the sequel we define the \emph{boundary size} as the semiperimeter of the bounding rectangle.

 We recursively construct a family of trees $\left\{T_n\right\}_{n\geq 2}$ as follows: $T_2$ is the star with 17 vertices, and for any $n>2$, we define $T_n$ as a tree with a vertex distinguished as root, to which $16n+1$ children are attached, with each of these children adjacent to a copy of $T_{n-1}$.

 \begin{figure}[htbp]

\centering

\includegraphics[width=11cm]{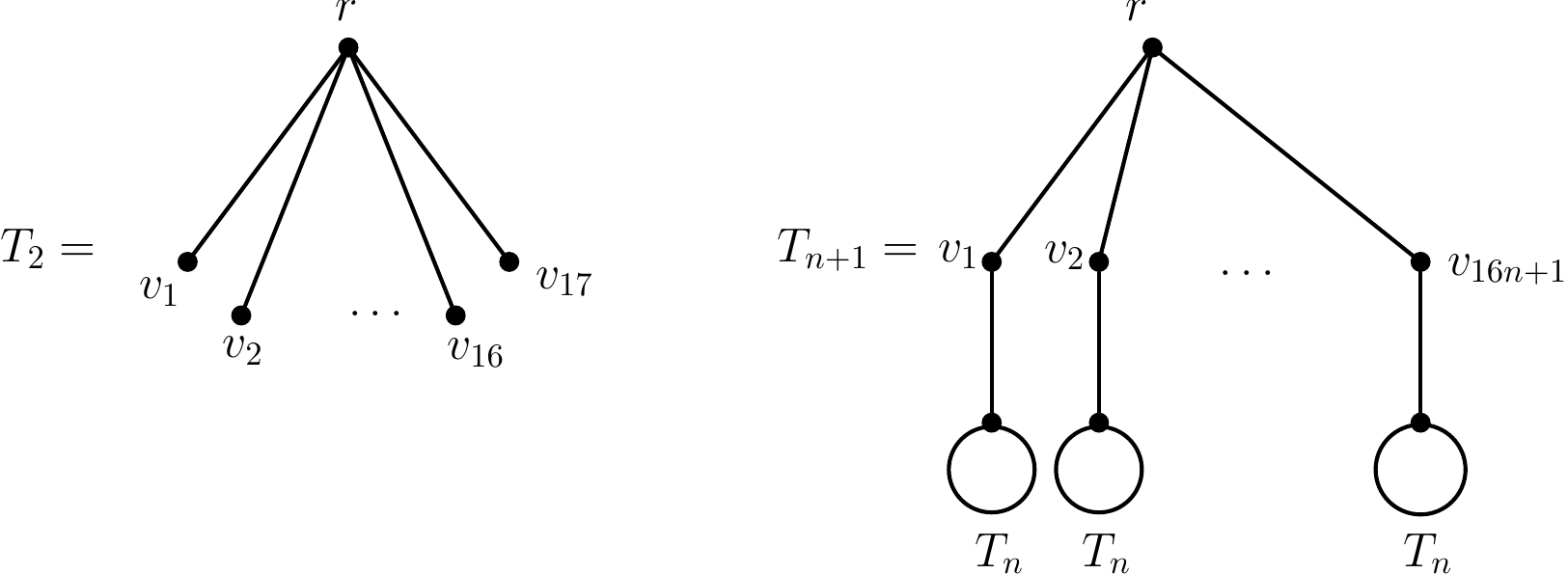}

\caption{Constructing the family of trees}

\label{tfamily}

\end {figure}

 \begin{thm} \label{thm:treeunbound}

 For all $n\geq 2$, a UGIG representation of $T_n$ needs a boundary size of at least $n$.

 \end{thm}


\begin{proof}

 We will proceed inductively:
 The base case $n=2$ is clear, as at least one unit is needed in both the horizontal and the vertical direction.

  For the induction step, we apply the box principle several times. Out of the $16n+1$ children of the root, at least $8n+1$ have their children either all above the root, or all below. Without loss of generality, we assume the latter is the case. Out of these $8n+1$ nodes, either at least $4n+1$ have children with an endpoint to the left of the root, or $4n+1$ have children to the right of the root. W.l.o.g. we consider the latter case. Restricting our analysis to these $4n+1$ nodes and their successors, we see that the paths of length two descending from the root form a nested structure, similar to the one in the following picture:

  \begin{figure}[htbp]

\centering

\includegraphics[width=4.5cm]{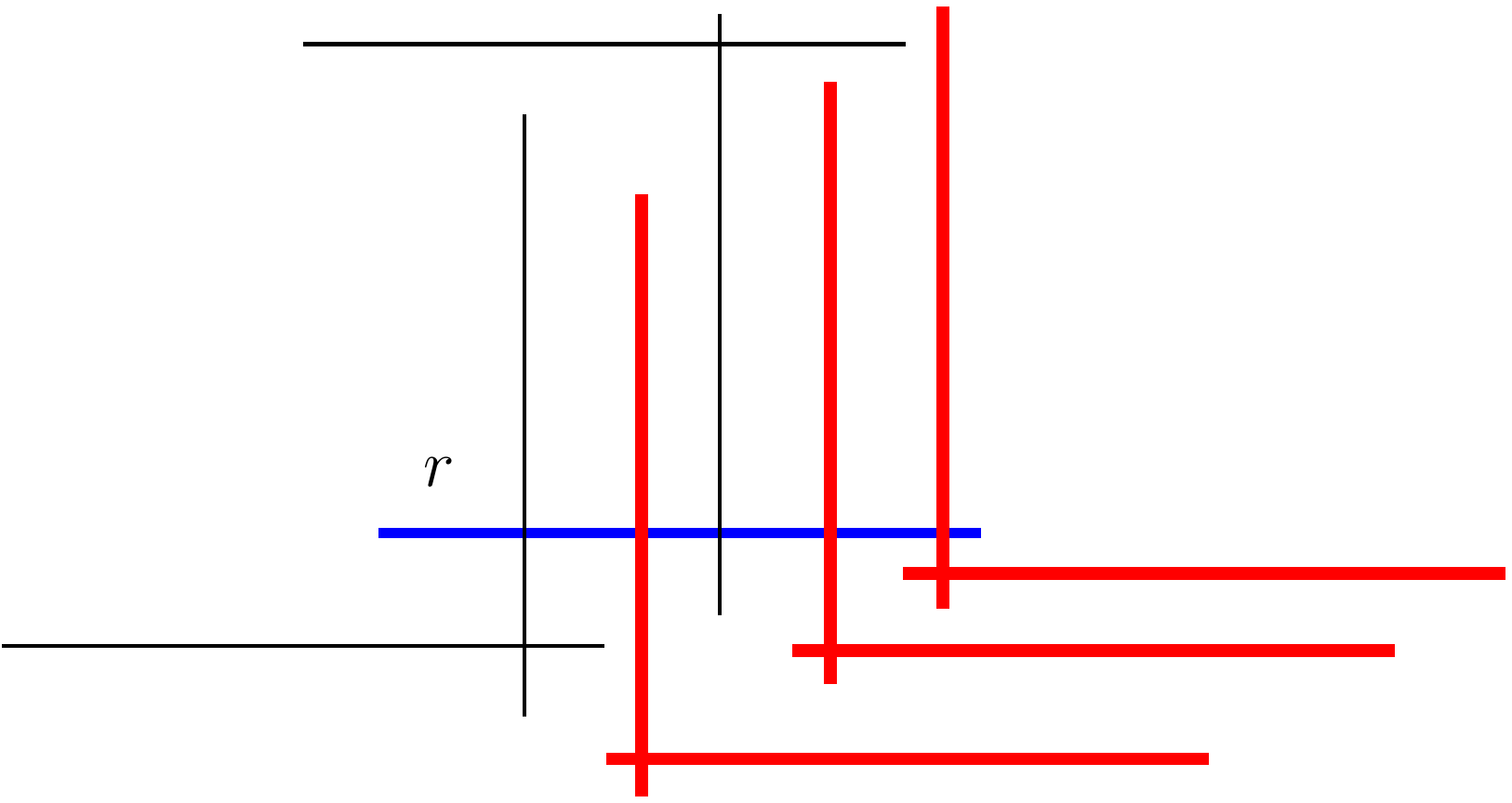}

\caption{Selecting the nested paths}

\label{selvert}

\end {figure}

Consider the second lowest lying child of the root and let $l$ be its descendant. We distinguish two cases:

\begin{itemize}

\item

At least one of the attached $T_n$s  does not go above  $l$. Then, from the induction hypothesis, we need a boundary size of at least $n$ for the copy of $T_n$ and an extra unit in the vertical direction that lies above $l$. Hence, in total, a rectangle of semiperimeter at least $n+1$ is necessary.

\item All attached $T_n$s go above $l$. This can only happen via a vertical segment to the right or to the left of $l$. We apply the box principle again, to conclude there exist $2n+1$ root children whose descending $T_n$s reach above $l$ from the right (the other case is analogous).

\begin{lem}

Let $u_1, \ldots u_{2n+1}$ be the top-down ordering of the descendants of the above $2n+1$vertices and $P_1, \ldots P_{2n+1}$ the corresponding paths that go above $l$. Then $y_{min} (horizontal(P_{2k})) < y_{min} (P_{2k-1})$  and $y_{min} (P_{2k+1}) < y_{min} (P_{2k})-1$ for all $k\geq 1$ (here $y_{min}(P_i)$ denotes the smallest $y$ coordinate of the path $P_i$). 

\end{lem}

\begin{proof}

Note that at every even step, a horizontal segment is added strictly below the current lowest vertical one.

The second statement follows from the fact that the currently lowest horizontal segment can only be passed by a path from left to right, if this path has a vertical segment lying strictly below, otherwise the unit segment condition would be violated. The previous step ensures the depth increasing by 1.

  \begin{figure}[htbp]

\centering

\includegraphics[width=6cm]{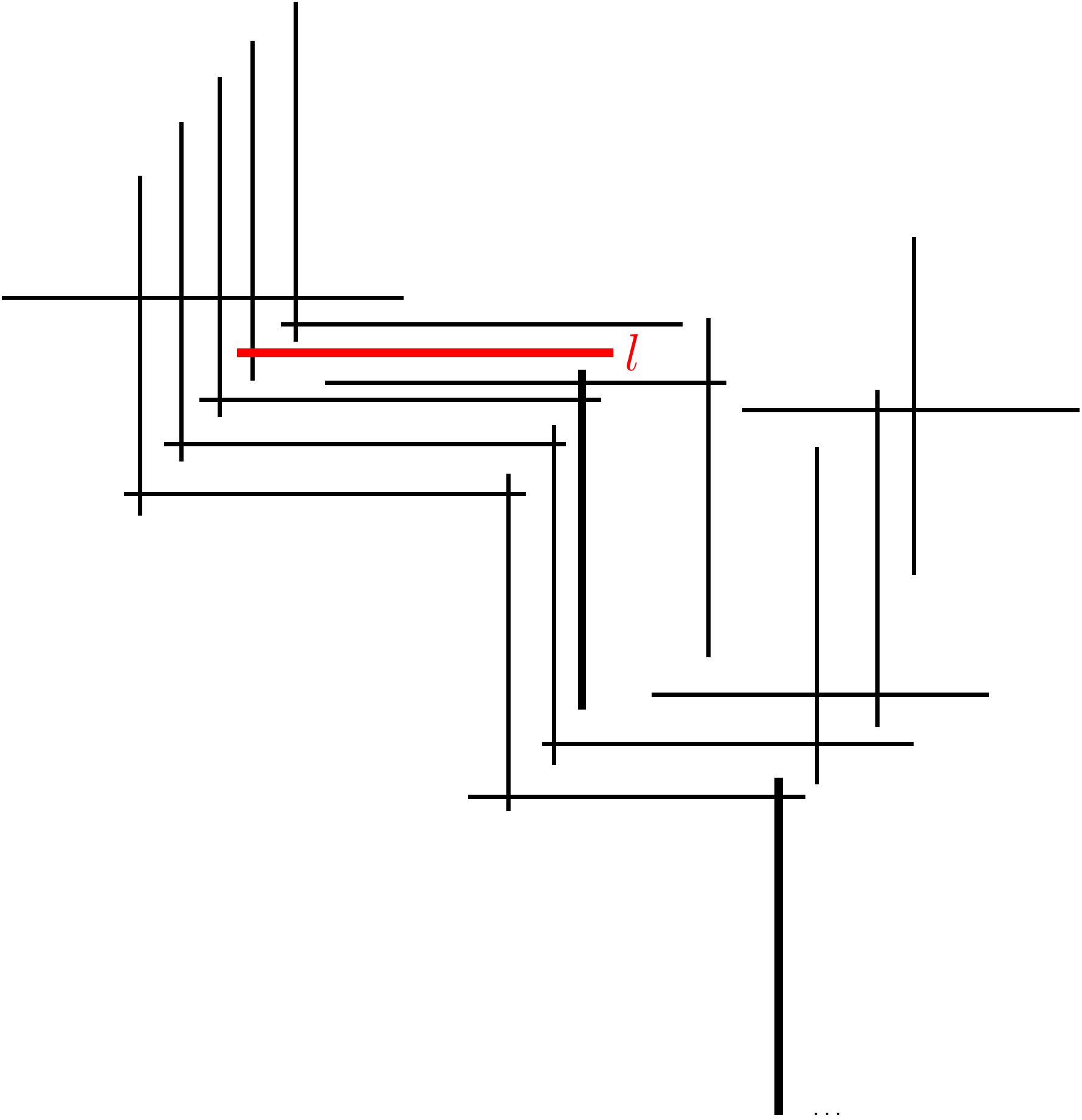}

\caption{Paths going above $l$}

\label{override}

\end {figure}

\end{proof}

The lemma above insures that the $2n+1$ paths would need at least extra $n$ units in height of the grid in order to go above $l$, hence again, proving the statement.

\end{itemize}

\end{proof}

\section{Recognition of unit grid intersection graphs} \label{sec:rec}

 This section is dedicated mainly to proving the following

 \begin{thm}\label{thm:nphard}

It is NP-hard to distinguish whether a given graph has a UGIG-representation or whether it does not even admit a representation by pseudosegments. Considering
any constant $k\in \mathbb{N}$, this holds even when restricting the analysis to graphs of girth at least $k$.

\end{thm}

\begin{cor}

The recognition problem is NP-hard for any class containing UGIGs
which is simultaneously contained in class of pseudosegment-graphs. For
such a class the recognition problem remains NP-hard even if restricted
to graphs with arbitrarily large girth.

\label{dusl:jednoduchy}

\end{cor}

\begin{cor}

The recognition problem is NP-complete for UGIGs as well as for
USEG graphs and the property holds even when
restricted to graphs with arbitrarily large girth. \label{dusl:npcompleteness}

\end{cor}

At the moment we prove  Theorem \ref{thm:nphard}. Corollary \ref{dusl:jednoduchy}
will follow directly. The remaining part of Corollary \ref{dusl:npcompleteness} will be shown later
in the text.

\begin{proof}

For graphs without the restriction on girth we may directly use the
reduction of Kratochv\'\i l and Pergel \cite{KP}. It is just necessary
to verify that the gadgets can be represented by orthogonal segments
of unit length.

\subsection{Planar 3-satisfiability}\label{ssec:3sat}

To show the result for graphs with arbitrary girth, we have to modify
this reduction: Although the clause-gadget does not require short cycles
(nor the variable gadget), this is not the case with the cross-over
gadget (even in the form of truth-splitter as presented in \cite{KP}).
Since seemingly no such arbitrary girth truth-splitter can be constructed, we will re-design
the reduction. Although the cross-over gadget was a cornerstone of all reductions based on Kratochv\'\i l's construction \cite{KratII}
so far, we use a similar approach without employing them, thus
having to re-design the variable gadget.

We reduce the question to solving Planar-3-connected-3-SAT(4), which was proven to be NP-hard
in \cite{KratII}. This problem asks whether a special type of boolean formula  in conjunctive
normal form is satisfiable.
Such a formula $\mathcal{F}$ must have the following properties:

\begin{itemize}

\item Each clause has at most three literals.

\item Each variable occurs at most four times.

\item $\mathcal{B}(\mathcal{F})$ is planar and 3-connected, where $\mathcal{B}(\mathcal{F})$ is the bipartite graph with one partition corresponding to the variables,
the other to the clauses, and an edge corresponds to
the occurence of a particular variable in a particular clause.

\end{itemize}

The general idea of the construction is that we consider this planar 3-connected graph and exploit the fact that it has
a unique planar embedding (up to the choice of an outer face)(REF needed). Each vertex
corresponding to a clause gets represented by a {\bf clause gadget}, whereas each
vertex corresponding to a variable gets represented by a {\bf variable gadget},
and finally, to each occurence we asign an {\bf occurence gadget}. Thus we obtain
a graph $G(\mathcal{F})$ that for a satisfiable instance (of the Planar-3-connected-3-SAT)
will be representable
by unit axis-aligned segments in a plane while for an unsatisfiable instance this graph
admits not even a pseudosegment representation. Moreover, the construction
of this graph permits us (by blowing the occurence- and clause gadgets
up) to obtain a graph of arbitrarily large (but constant) girth. In the following we will
describe these three gadgets:

\subsection{The gadgets} \label{ssec:gadget}

The \textbf{occurence gadget} is formed by a pair of non-intersecting paths $P_1, P_2$.
These paths start from the variable gadget and finish in the clause gadget.
The truth-assignment is given by the left-right orientation of these paths,
i.e., the occurence is true iff $P_1$ is (clockwise) "to the left" of $P_2$. This
way of representing occurences is usual in the reductions of this family.

The \textbf{clause gadget} gets reused from \cite{KP}. To make the article self-contained, we show it in the Figure \ref{clause:gadget}.

\begin{figure}[htbp]

\centering

\includegraphics[width=5cm]{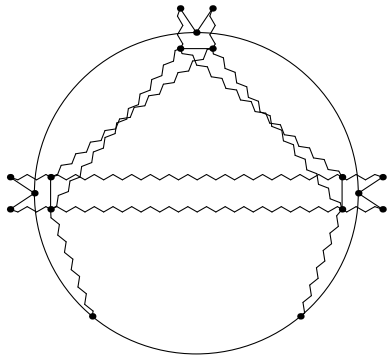}

\caption{The clause gadget}

\label{clause:gadget}

\end{figure}

 Straight segments represent one vertex, jigsaw curves depict arbitrarily long
paths whose length depends on the required girth. The surrounding cycle can
be extended only between individual pairs of paths, that correspond to the occurence gadgets entering the clause.

By our convention, the occurence-gadgets are entering it
from the left, from the right and from the top.

 In \cite{KP} it is proven that
this gadget cannot be represented even by pseudosegments when all three
literals are false. We now have to show that it can be represented by unit
segments in all other cases, which can be seen in Figure \ref{fig:notfff} below. 

\begin{figure}[htbp]

\centering

\includegraphics[scale=0.7]{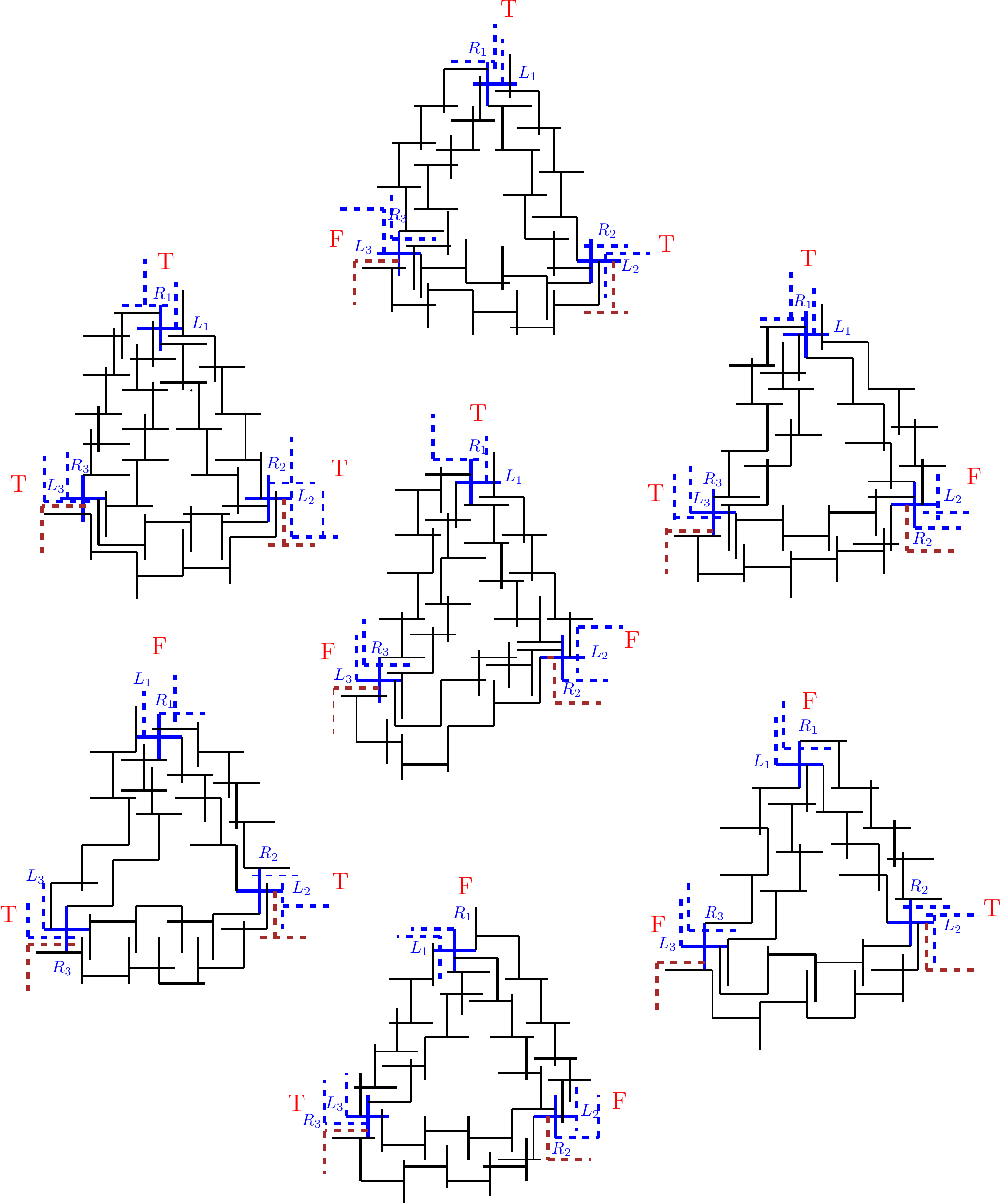}

\caption{UGIG representations of the clause gadget in the seven cases of satisfiability}

\label{fig:notfff}

\end{figure}

What is innovative about this construction, is the \textbf{variable gadget}
(that does not need a cross-over).
The aim of the variable gadget is to synchronize the truth-assignment
of the literals containing  the same variable. This gadget must allow sending three or four occurences of one variable to the corresponding clauses, while keeping the orientation of all of them consistent. Note that because
we are reducing not to simply 3-SAT(4) but Planar-3-connected-3-SAT(4),
due to 3-connectivity each variable must occur at least 3 times.

This variable gadget is very simple. It consists of two 
adjacent vertices $a, b$ (each variable has its own pair of vertices).
Out of each pair of paths of an occurence gadget,
one path gets attached to $a$ and the other to $b$. Which path gets
connected to which vertex is determined by the circular order of the
paths in the embedding of the 3-connected graph. At this point, please, note that we know
the circular-ordering in which the occurence-gadgets shall stem
from the variable gadget.

The vertices $a$ and $b$ are represented
by two crossing segments. Individual pairs of paths (representing
individual occurences) start from this cross and enter the corresponding
clause gadgets. Considering a variable that occurs four times, in
the first pair (representing the first occurence) the "left" path gets
attached to $a$, the "right" path to $b$. The same situation happens with
the third occurence while in the second and fourth occurence (with respect
to the circular ordering) the "left" is incident to $b$ and the "right"
path to $a$. Such a representation is depicted in Figure \ref{variable} below.

\begin{figure}[htbp]
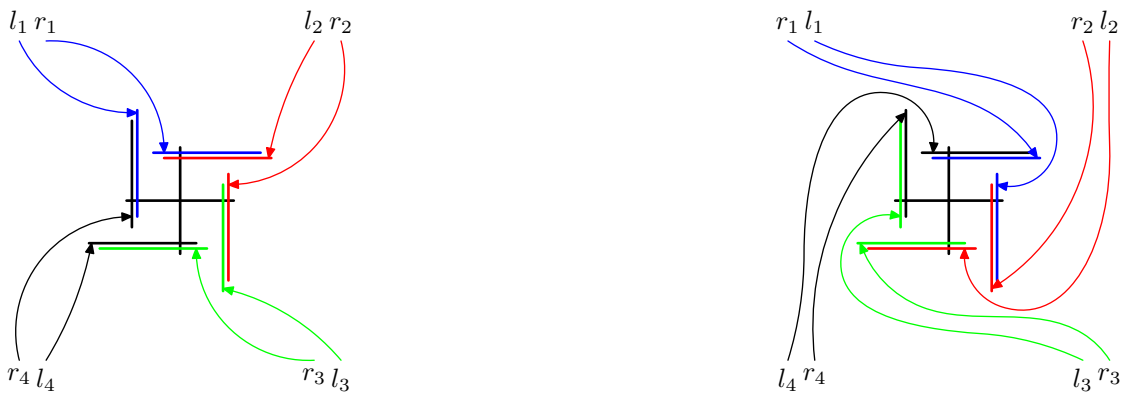


\includegraphics{obrazky.1}\hfill \includegraphics{obrazky.2}

\caption{The variable gadget: true(on the left) and false (on the right) assignment}

\label{variable}

\end{figure}

If the variable is assigned "true", the "left" paths
are to the left from the "right" ones), and the converse. The figure
shows only how the occurence-gadgets stem from the variable gadget. Note that it is possible to represent the other endpoints of the paths as segments that enter the corresponding clause gadgets. Also note that all four occurences must keep consistent truth-value. Here, the horizontal segment represents the vertex $a$, the vertical one corresponds to $b$.

Each individual occurence "blocks" the visibility
of the intersection point of the segments of $a$ and $b$  from one quadrant. Therefore there is enough space for
{\em exactly}  four pairs of paths. The planarity constraint also ensures that the orientation of one pair forces
the orientation of all  the others. 
If some variable occurs only three times (a lower number of occurences is
impossible due to $3$-connectivity of the given graph $G(\mathcal{F})$), we create a {\it dummy
occurence}. This {\it dummy occurence} behaves like an occurence gadget,
but finishes in the middle of a prescribed face. To prescribe in what
face this
gadget gets represented (to avoid distorting the left-right orientation),
we attach it to the neighboring occurence gadgets (see Figure \ref{dummy:occurence}).

\begin{figure}[htbp]
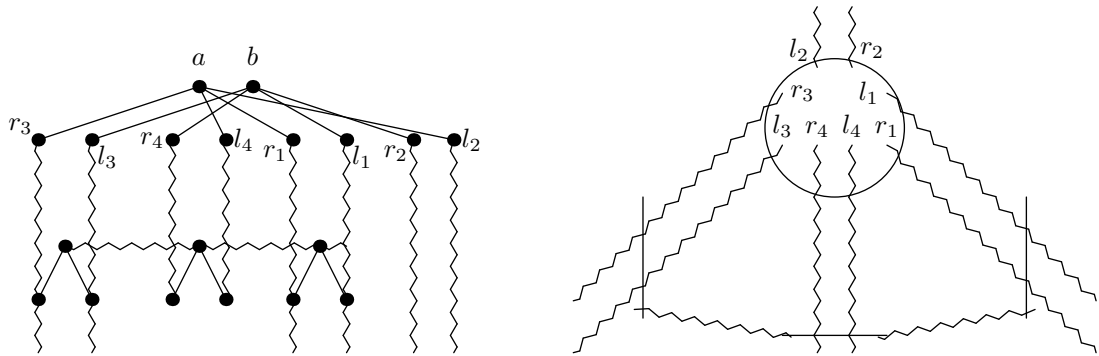


\begin{minipage}[b]{0.5\linewidth}
\includegraphics{obrazky.21}
\end{minipage}
\begin{minipage}[b]{0.5\linewidth}
\includegraphics{obrazky.20}
\end{minipage}

\caption{A dummy gadget (left) with an appropriate UGIG representation (right). note that in the representation we omitted the details inside the circle where
the paths should be appropriately connected to segments representing $a$ and $b$.}

\label{dummy:occurence}

\end{figure}

Note that the fourth (dummy) occurence starts as in the description of the variable gadget, and in order to keep it inside a prescribed face (to keep the consistence of the truth-assignment) we attach it to the neighboring occurences. 
 The length of these
paths depend on a prescribed girth. Note that inside the circle (on the right picture) we omitted the representation of the variable-gadget.

Note that the whole argumentation does not depend on geometrical properties( we do not use the fact that we are permitted to use only two directions
to represent the segments). Thus by this construction
we create either a unit grid intersection graph in case of formula satisfiability, or not even a PSEG-graph, hence these two
classes cannot be efficiently separated (if P$\not=$NP). Therefore
the recognition of any class containing UGIG-graphs that is simultaneously
contained in the class of PSEG-graphs is also NP-hard.

\end{proof}

\subsection{A polynomial certificate}

Now we show the remaining part of the proof of Corollary \ref{dusl:npcompleteness}. Although we could use the canonical construction with Cartesian coordinate-system, as we are interested in optimal representations, we show that even a smaller certificate suffices:

\begin{thm}

 The recognition problem for  UGIG and USEG graphs is in NP, i.e.,
there exists a certificate of size $O(n\log n)$ for which we can verify (in polynomial time) that it witnesses the representation.

\end{thm}

\begin{proof}

We use a method similar to the proof of equivalence for unit interval
graphs and proper interval graphs. Thus we have to encode the ordering of start- and end-points for individual segments with respect to individual axes. Name of each segment can be encoded by (say binary) number with size logarithmic w. r. t. its value. Details in the Appendix.

\end{proof}

 \section* {Acknowledgements}

 We would like to thank Prof. Stefan Felsner for his contributions to the proof of Theorem \ref{thm:orgineps} and Theorem \ref{thm:treeunbound}.

\section*{Open problems}


It has been unknown for several years how difficult it is to recongnize and/or describe orthogonal ray graphs. This becomes even more interesting in the context of this graph class lying between the NP-complete UGIGs and the 
quadratic orthogonal ray graphs. Recently, G. Mertzios, I. M. and S. Felsner have shown that there exists a polynomial recognition algorithm for ORGs, in the case the vertices have prescribed ray orientations. \cite{FMM} 

One can also ask about the description of hybrid bipartite graph classes, with one partition representable by (unit) segments and the other by rays. Another open problem is whether the recognition of unit segment graphs (without restriction on directions) is in NP. We show that it is NP-hard and one could observe (from Pythagorean triangles) that each direction can be approximated by a unit segment with rational endpoints, but it may happen that for describing an $n$ vertex graph, we need, e.g., exponential precision for individual coordinates.

 \bibliographystyle{plain}

\bibliography{ugig3}

\end{document}